\documentclass[10pt]{amsart}

\usepackage{amssymb,amsmath,amscd,amsthm,amsfonts,wasysym,mathrsfs,amsxtra}
\input xy
\xyoption{all}

\newcommand{\MM}{\mathcal{M}}

\newcommand{\EE}{\mathscr{E}}

\newcommand{\Cfield}{\mathbb{C}}

\newcommand{\rank}{\textnormal{rank}\,}

\newcommand{\Ext}{\textnormal{Ext}}

\newcommand{\arinj}{\ar@{^{(}->}}
\newcommand{\arsurj}{\ar@{->>}}

\newcommand{\MMs}{\mathcal{M}^{\sigma ss}}
\newcommand{\MMm}{\mathcal{M}^{\mu ss}}
\newcommand{\MMml}{\mathcal{M}^{\mu ss, lf}}
\newcommand{\MMss}{\mathcal{M}^{ss}}

\newtheorem{theorem}{Theorem}[section]

\newtheorem{lemma}[theorem]{Lemma}

\newtheorem{proposition}[theorem]{Proposition}
\newtheorem{coro}[theorem]{Corollary}

\theoremstyle{definition}

\theoremstyle{remark}
\newtheorem{remark}[theorem]{Remark}

\numberwithin{equation}{section}

\begin{document}

\title{On some moduli of complexes on K3 surfaces}

\author[Jason Lo]{Jason Lo}
\address{Department of Mathematics, University of Missouri,
Columbia, MO 65211, USA} \email{locc@missouri.edu}


\maketitle

\begin{abstract}
We consider moduli stacks of Bridgeland semistable objects that previously had only set-theoretic identifications with Uhlenbeck compactification spaces in \cite{LQ}.  On a K3 surface $X$, we give examples where such a moduli stack  is isomorphic to a moduli stack of slope semistable locally free sheaves on the Fourier-Mukai partner $\hat{X}$.  This yields a morphism from the stack of Bridgeland semistable objects to a projective  scheme, which induces an injection on closed points.  It also allows us to extend a theorem of Bruzzo-Maciocia on Hilbert schemes to a statement on moduli of complexes.
\end{abstract}


\section{Introduction}

Fourier-Mukai transforms have been used extensively to understand relations between various moduli spaces of sheaves.  More recently, they have also been used to understand  relations between moduli of complexes  and moduli of sheaves.  For example, Bernardara-Hein considered complexes on elliptic surfaces \cite{BH}, and Hein-Ploog had examples on elliptic K3 surfaces \cite{HP}.  Maciocia-Meachan considered rank-one Bridgeland stable objects on Abelian surfaces \cite{MM}, while Minamide-Yanagida-Yoshioka considered the relations between Bridgeland semistable complexes and Gieseker semistable sheaves on Abelian and K3 surfaces in \cite{MYY, MYY2}.

In \cite[Theorem 5.4]{LQ}, it was shown that, on a polarised smooth projective surface $(X,\omega)$, if the rank $r$ and degree $c_1\omega$ of the objects are coprime and the slope $c_1\omega/r$ of the objects is not equal to $\beta \omega$,  then all the polynomial semistable objects are stable, and the moduli stack of polynomial stable objects is  isomorphic to the moduli scheme of $\mu_\omega$-stable sheaves; here, $\beta \in \text{Num}(X)_{\mathbb Q}$ is another parameter in the definition of polynomial stability.  And when $\mu_\omega$ is equal to $\beta \omega$ (part (iii) of \cite[Theorem 5.4]{LQ}), it was only shown that there is a set-theoretic bijection between the closed points of the moduli of polynomial semistable objects, and the Uhlenbeck compactification space of $\mu_\omega$-stable locally free sheaves.  In this case, the polynomial semistable objects coincide with the Bridgeland semistable objects of phase 1.

In this paper, we consider the latter case above, i.e.\ we consider the moduli of Bridgeland semistable objects of slope  $\mu_\omega = \beta \omega$.   If $X$ is a K3 surface, we show that the moduli $\MMs$ of these objects is isomorphic to a moduli of $\mu$-semistable locally free sheaves on the Fourier-Mukai partner of $X$ under an extra assumption (Theorem \ref{theorem1}).  This yields a morphism of stacks from $\MMs$ to a projective scheme, which induces an injection on the closed points (Corollary \ref{coro1}).  This way, we obtain an algebro-geometric alternative to the statement in part (iii) of \cite[Theorem 5.4]{LQ}.

Theorem \ref{theorem1} also extends a theorem of Bruzzo-Maciocia's \cite[Theorem 0.1]{BM}: they constructed  an isomorphism between the Hilbert scheme of points and a moduli space of Gieseker stable locally free sheaves on a reflexive K3 surface, which now extends to an isomorphism between a moduli stack of Bridgeland semistable objects and a moduli stack of $\mu$-semistable locally free sheaves on any K3 surface.  We explain this along with the details of other consequences of Theorem \ref{theorem1} in Section \ref{section-app}.

The main ingredients in our proof of Theorem \ref{theorem1} are Huybrechts' results on derived equivalent K3 surfaces \cite{Huy06}.

\subsection{Preliminaries}

For any scheme $X$, we write $D(X)$ to denote the bounded derived category of coherent sheaves on $X$.  For an object $E \in D(X)$, we write $H^i(E)$ to denote the degree-$i$ cohomology of $E$.

If $(X,\omega)$ is a polarised smooth projective variety and $E \in D(X)$, we write $\mu_\omega (E)$ to denote $c_1(E)\omega/\rank (E)$.  Furthermore, if $\beta \in \text{Num}(X)_\mathbb{Q}$, then we write $\mathcal A (\beta,\omega)$ to denote the Abelian subcategory of $D(X)$ consisting of 2-term complexes $E \in D(X)$ such that:
\begin{itemize}
\item all the Harder-Narasimhan factors of $H^0(E)$ have slope $\mu_\omega > \beta \omega$;
\item $H^{-1}(E)$ is torsion-free, and all its Harder-Narasimhan factors have slope $\mu_\omega \leq \beta \omega$;
\item $H^i(E)=0$ for all $i \neq -1, 0$.
\end{itemize}
The Abelian category $\mathcal A (\beta,\omega)$ is the heart of a bounded t-structure on $D(X)$ that is obtained from the standard t-structure from tilting.

The only moduli of polynomial semistable objects discussed in the rest of this paper is as in part (iii) of \cite[Theorem 5.4]{LQ}.  In this case, polynomial stability coincides with a Bridgeland stability $\sigma$ on $D(X)$, and the polynomial semistable objects are exactly the $\sigma$-semistable objects of phase 1  in $\mathcal A (\beta,\omega)$.  Therefore, from now on, we will refer to these semistable objects only as Bridgeland semistable objects.  Note that $\beta, \omega$ are part of the definition of the polynomial/Bridgeland stability here.    And for fixed $\beta, \omega$, the central charge on the heart $\mathcal A(\beta, \omega)$ that we use is
\begin{align}
  Z (E) &:= - \int_X e^{-(\beta+i\omega)} \cdot ch(E) \text{\quad for $E \in \mathcal A (\beta, \omega)$} \notag\\
  &= \text{rk}(E)\frac{\omega^2}{2} + i (c_1(E)\omega - \text{rk}(E)\beta \omega) + c(E) \label{eq12}
\end{align}
where $c(E) := -ch_2(E)+c_1(E)\beta - \text{rk}(E) \beta^2/2$.  The phase of a nonzero object $E \in \mathcal A(\beta, \omega)$ is defined to be the real number $\phi (E) \in (0,1]$ satisfying
\[
  Z(E) \in \mathbb{R}_{>0}\cdot e^{i \phi (E)}.
\]
From \eqref{eq12}, it is clear that an object $E \in \mathcal A (\beta, \omega)$ of positive rank and of slope $\mu_{\omega} (E) := c_1(E)\omega/\text{rk}(E) = \beta \omega$ has phase 1 (maximal phase).  The reader may refer to \cite[Section 2]{LQ} for details of the definitions of these stability conditions.

All the schemes will be over $k=\Cfield$.

\subsection{Acknowledgments}
The author would like to thank Arend Bayer, Dan Edidin, Zhenbo Qin and Ziyu Zhang for  helpful discussions.

\section{K3 surfaces}\label{section-K3}

Throughout this section, suppose $X$ and $\hat{X}$ are two non-isomorphic, derived equivalent K3 surfaces.  By \cite[Proposition 4.1]{Huy06}, $\hat{X}$ is isomorphic to a fine moduli space of $\mu_\omega$-stable locally free sheaves on $X$, with respect to some ample class $\omega \in \text{Num}(X)_\mathbb{Q}$.   Suppose the $\mu_\omega$-stable sheaves parametrised by $\hat{X}$ have slope $\mu_\omega =\beta \omega$ for some $\beta \in \text{Num}(X)_\mathbb{Q}$.    If we let  $\EE$ denote the universal family on $X \times \hat{X}$, then the Fourier-Mukai transform
\[
\Psi : D(X) \overset{\thicksim}{\to} D(\hat{X})
 \]
 with kernel $\EE$  induces an equivalence from the Abelian category $\mathcal A (\beta, \omega)$ to another Abelian category $\mathcal A(\hat{\beta}, \hat{\omega})$ for some $\hat{\beta}, \hat{\omega} \in \text{Num}(\hat{X})_{\mathbb{Q}}$ where $\hat{\omega}$ is ample \cite[Proposition 5.2]{Huy06}.

Let $\sigma, \hat{\sigma}$ be the Bridgeland stabilities on $X, \hat{X}$ defined using $\beta, \omega$ and $\hat{\beta}, \hat{\omega}$, respectively.

\begin{lemma}\label{lemma3}
The Fourier-Mukai transform $\Psi$ takes a $\sigma$-semistable object in $\mathcal A (\beta,\omega)$ of slope $\mu_\omega (E) = \beta \omega$ to a $\hat{\sigma}$-semistable object in $\mathcal A (\hat{\beta},\hat{\omega})$ of slope $\mu_{\hat{\omega}} = \hat{\beta} \hat{\omega}$.
\end{lemma}

\begin{proof}
Take any $\sigma$-semistable object $E \in \mathcal A (\beta,\omega)$ of slope $\mu_\omega (E)=\beta\omega$ (so $H^{-1}(E)$ is necessarily nonzero).  From \cite[Lemma 4.2(c)]{BayerPBSC}, we know $H^{-1}(E)$ is $\mu_\omega$-semistable of slope $\beta \omega$ and $H^0(E)$ is a 0-dimensional sheaf.

Consider the Jordan-H\"{o}lder filtration of $H^{-1}(E)$ with respect to $\mu_\omega$-stability:
\begin{gather}\label{eq9}
  0 \subsetneq F_1 \subsetneq F_2 \subsetneq \cdots \subsetneq F_m = H^{-1}(E),
\end{gather}
where each $G_i := F_i/F_{i-1}$ is $\mu_\omega$-stable with slope $\beta \omega$.  And then, for each $i$, we have the short exact sequence of coherent sheaves
\begin{equation}\label{eqn-extension}
  0 \to G_i \to G_i^{\ast \ast} \to T_i \to 0
\end{equation}
where $G_i^{\ast \ast}$ is a $\mu_\omega$-stable locally free sheaf of slope $\beta \omega$, and $T_i$ is a 0-dimensional sheaf $T_i$; this gives an exact triangle
\begin{equation}\label{eq8}
T_i \to G_i [1] \to G_i^{\ast \ast}[1] \to T_i [1]
\end{equation}
for each $i$.

By \cite[Proposition 2.2]{Huy06}, the minimal objects in $\mathcal A (\beta,\omega)$ are of the following forms:
\begin{itemize}
\item the skyscraper sheaf $k(x)$ at a closed point  $x \in X$, or
\item $F[1]$ where $F$ is a $\mu_\omega$-stable locally free sheaf with $\mu_\omega (F)=\beta \omega$.
\end{itemize}
Therefore, using the canonical exact triangle
\[
  H^{-1}(E)[1] \to E \to H^0(E) \to H^{-1}(E)[2]
\]
in $D(X)$, together with the filtration \eqref{eq9} and the exact triangles \eqref{eq8}, we can construct $E$ as a series of extensions by minimal objects in $\mathcal A (\beta,\omega)$.  Since $\Psi$ induces an equivalence between $\mathcal A(\beta,\omega)$ and $\mathcal A (\hat{\beta},\hat{\omega})$, it takes minimal objects in  $\mathcal A(\beta,\omega)$ to minimal objects in $\mathcal A (\hat{\beta},\hat{\omega})$.  And any object in $\mathcal A (\hat{\beta},\hat{\omega})$ built from a series of extensions by minimal objects is $\hat{\sigma}$-semistable with slope $\mu_{\hat{\omega}}=\hat{\beta}\hat{\omega}$. This completes the proof of the lemma.
\end{proof}

Note that, for any $y \in \hat{X}$, the Chern character of the fibre $\EE_y$ of the universal family $\EE$  is independent of $y$.

\begin{proposition}\label{pro1}
For any $\sigma$-semistable object $E \in \mathcal A(\beta,\omega)$ of Chern character $ch=(-r,-\delta,n)$, with $r>0$ and satisfying:
\begin{itemize}
\item[(a)] $\delta \omega /r = \beta \omega$, i.e.\ the $\mu_\omega$-slope of $E$ is $\beta \omega$, and
\item[(b)] for any of the Jordan-H\"{o}lder factor $G$ of $H^{-1}(E)$ with respect to $\mu_\omega$-stability, we have $ch(G^\ast) \neq ch(\EE_y)$,
\end{itemize}
the transform $\Psi (E)$ is a $\mu_{\hat{\omega}}$-semistable locally free sheaf of slope $\mu_{\hat{\omega}} = \hat{\beta}\hat{\omega}$.
\end{proposition}

In Lemma \ref{lemma4} below, we give two numerical conditions where condition (b) in Proposition \ref{pro1} is satisfied.

\begin{proof}
We use the same notation as in the proof of Lemma \ref{lemma3}, so that the $G_i$ denote the Jordan-H\"{o}lder factors of $H^{-1}(E)$.  For each $i$, the sheaf $G_i^\ast$ is reflexive, hence locally free, and so $ch_2 (G_i^\ast) = ch_2 (G_i^{\ast \ast})$.  By condition (b) in the hypothesis, none of the $G_i^{\ast \ast}$ is isomorphic to any $\EE_y$.  Since each $G_i^{\ast \ast}$ is a $\mu_\omega$-stable locally free sheaf with slope $\mu_\omega =\beta \omega$, by \cite[Proposition 7.1]{Huy06}, each of the transforms $\Psi (G_i^{\ast \ast})$ is a $\mu_{\hat{\omega}}$-stable locally free sheaf of slope $\hat{\beta}\hat{\omega}$ (up to shift).  On the other hand, every 0-dimensional sheaf on $X$ is mapped by $\Psi$ to a $\mu_{\hat{\omega}}$-stable locally free sheaf of slope $\hat{\beta}\hat{\omega}$.  Therefore, using the exact triangle
\[
 H^{-1}(E)[1] \to E \to H^0(E) \to H^{-1}(E)[2],
\]
along with the filtration \eqref{eq9} and the triangles \eqref{eq8}, we can construct $\Psi (E)$  by a series of extensions of $\mu_{\hat{\omega}}$-stable locally free sheaves of slope $\hat{\beta} \hat{\omega}$, implying $\Psi (E)$ itself is a $\mu_{\hat{\omega}}$-stable locally free sheaves of slope $\hat{\beta} \hat{\omega}$, sitting at degree $-1$.
\end{proof}

\begin{lemma}\label{lemma4}
With the same notation as in Proposition \ref{pro1}, condition (b) of Proposition \ref{pro1} is satisfied when:
\begin{itemize}
\item[(i)] $0 < r  < \rank (\EE_y)$, or
\item[(ii)]
\begin{gather}
ch_2(\EE_y) < -  \left| n + (r-1) \frac{(\delta\omega)^2}{2\omega^2}\right|.\label{eq1}
\end{gather}
\end{itemize}
\end{lemma}

\begin{proof}
That the inequality in (i) implies condition (b) is clear.  Now, suppose $ch(E)=(-r,-\delta,n)$ satisfies the inequality \eqref{eq1} in (ii).  Using the same notation as in the proof of Proposition \ref{pro1}, we show  that for each $i$, we have $ch_2(G_i^{\ast \ast}) \neq ch_2 (\EE_y)$:

  For any $i$, write $f_i := c_1(G_i)$ and $g_i := ch_2 (G_i)$.  Then for all $i$, we have $\mu_\omega (H^{-1}(E)) = \mu_\omega (G_i)$, i.e.\ $\delta \omega /r = f_i \omega/\text{rk}(G_i)$, i.e.\ $f_i \omega = \delta \omega \cdot \text{rk}(G_i)/r$.  By \cite[Lemma 3.7(i)]{LQ}, we have
  \[
  f_i^2 \leq \frac{\left(\delta \omega \cdot \frac{\text{rk}(G_i)}{r}\right)^2}{\omega^2} \leq \frac{(\delta \omega)^2}{\omega^2}.
  \]
On the other hand, Bogomolov's inequality on each $G_i$ gives us $g_i \leq f_i^2/(2\, \text{rk} (G_i))$.  So overall, we have
\begin{equation}\label{eq2}
g_i \leq \frac{f_i^2}{2\,\text{rk} (G_i)} \leq \frac{(\delta \omega)^2}{2\omega^2}.
\end{equation}
Also, since $ch_2(H^0(E)) \geq 0$ and $n = -ch_2(H^{-1}(E))+ch_2(H^0(E))$, we have
\begin{equation}\label{eq3}
-n \leq ch_2(H^{-1}(E)) = \sum_{i=1}^m g_i.
\end{equation}

The inequality \eqref{eq2} and equation \eqref{eq3} together give, for all $i$,
\begin{equation*}
  (-|g_i|) + (m-1) \frac{(\delta\omega)^2}{2\omega^2} \geq -n,
\end{equation*}
i.e.\
\[
  |g_i| \leq n + (m-1) \frac{(\delta\omega)^2}{2\omega^2}.
\]
Since $m \leq r$, we have $|g_i| \leq n + (r-1) \frac{(\delta\omega)^2}{2\omega^2}$, hence
\begin{equation}\label{eq4}
 g_i \geq  - \left| n + (r-1) \frac{(\delta\omega)^2}{2\omega^2} \right|
\end{equation}
for all $i$.  Therefore, by the inequality \eqref{eq1}, $ch_2(G_i^\ast)=ch_2 (G_i^{\ast\ast}) \geq ch_2 (G_i) =g_i> ch_2 (\EE_y)$ for any $i$ and $y\in \hat{X}$.  Hence $ch(G_i^{\ast\ast}) \neq ch (\EE_y)$ for any $i$.
\end{proof}

\begin{remark}\label{remark1}
In particular, if $r=1$ (i.e.\ $E$ is of rank 1), or $\delta\omega=0$, then the inequality \eqref{eq1} reduces to $ch_2 (\EE_y) < -|n|$.
\end{remark}

\section{Moduli stacks}\label{section3}

Let $X$ be any smooth projective K3 surface, and let $\sigma$ be a Bridgeland stability as in Section \ref{section-K3}.

%


Lieblich \cite{Lieblich} constructed an Artin stack $\MM$ of objects $E \in D(X)$ satisfying $\Ext^{<0}(E,E)=0$.  We have various open substacks of $\MM$:
\begin{itemize}
\item $\MMs_{X,ch}$, the Artin stack of $\sigma$-semistable objects of Chern character $ch$ on $X$, constructed by Toda \cite{Toda1};
\item $\MMm_{X,ch}$, the Artin stack of $\mu_\omega$-semistable torsion-free sheaves of Chern character $ch$ on $X$;
\item $\MMml_{X,ch}$, the Artin stack of $\mu_\omega$-semistable locally free sheaves of Chern character $ch$ on $X$, which is a substack of $\MMm_{X,ch}$.
\end{itemize}

We have:

 \begin{theorem}\label{theorem1}
Suppose all the objects in $\MMs_{X,ch}$ satisfy the conditions of Proposition \ref{pro1}.  Then the Fourier-Mukai transform $\Psi : D(X) \to D(\hat{X})$ induces an isomorphism of Artin stacks
\[
\Psi_{\MM} :  \MMs_{X,ch} \to \MM^{\mu ss,lf}_{\hat{X},\hat{ch}}
 \]
(where $\hat{ch}$ denotes the image of $ch$ under $\Psi^{ch}$).
 \end{theorem}

Let us set up some notation for the proof of the theorem.  The equivalence $\Psi : D(X) \to D(\hat{X})$ has a quasi-inverse $\Phi : D(\hat{X}) \to D(X)$ that is also a Fourier-Mukai transform, with kernel $\EE^\ast [1]$.  For any scheme $S \to k$, we write $X_S$ to denote $X \times_k S$.  Then we have relative Fourier-Mukai transforms $\Psi_S : D(X_S) \to D(\hat{X}_S)$ and $\Phi_S : D(\hat{X}_S) \to D(X_S)$ induced by $\Psi, \Phi$, respectively.  For any closed point $s \in S$, we write $X_s$ to denote the fibre of $X_S$ over $s$, write $j_s$ to denote the closed immersion $X_s \hookrightarrow X_S$ or $\hat{X}_s \hookrightarrow \hat{X}_S$, and write $\Psi_s, \Phi_s$ to denote the Fourier-Mukai transforms between the derived categories of $X_s, \hat{X}_s$ induced by $\Psi_S,\Phi_S$, respectively.  See \cite[Section 1.2.1]{FMNT} for more on relatively Fourier-Mukai transforms.

\begin{proof}
Take any scheme $S$ over $k$.  Let $E_S$ be an $S$-flat family of $\sigma$-semistable objects on $X$ of Chern character $ch$.  By the base change formula \cite[(1.15)]{FMNT}
\begin{gather}\label{eq10}
  Lj_s^\ast \Psi_S (E_S) \cong \Psi_s (Lj_s^\ast E_S)
\end{gather}
and Proposition \ref{pro1}, we see that  $\Psi_S (E_S)$ is an $S$-flat family of $\mu$-semistable locally free sheaves on $\hat{X}$.  This induces a morphism of stacks $\Psi_\MM : \MMs_{X,ch} \to  \MM^{\mu ss,lf}_{\hat{X},\hat{ch}}$.

Conversely, suppose that $F_S$ is an $S$-flat family of $\mu$-semistable locally free sheaves of Chern character $\hat{ch}$ on $\hat{X}$.   Since each fibre of $F_S$ over $S$ is a $\hat{\sigma}$-semistable object of slope $\mu_{\hat{\omega}}=\hat{\beta}\hat{\omega}$, by Lemma \ref{lemma3} and the base change formula above (both assertions are symmetric in $X$ and $\hat{X}$), we see that $\Phi_S$ takes $F_S$ to an $S$-flat family of $\sigma$-semistable objects of Chern character $ch$.  This induces a morphism of stacks $\Phi_\MM :  \MM^{\mu ss,lf}_{\hat{X},\hat{ch}} \to \MMs_{X,ch}$.  Since $\Psi$ and $\Phi$ are quasi-inverses to each other, so are $\Psi_S$ and $\Phi_S$.  Hence $\Psi_\MM$ is an isomorphism of stacks.
%
%
\end{proof}

Given any smooth curve $C$ on $\hat{X}$, let $\MMss_{C,(r,d)}$ denote the moduli stack of $\mu$-semistable (equivalently, Gieseker semistable) torsion-free sheaves on $C$ of rank $r$ and degree $d$.  We know that $\MMss_{C,(r,d)}$ admits a good  moduli space
\[
  \gamma : \MMss_{C,(r,d)} \to M^{ ss}_{C,(r,d)}
\]
in the sense of Alper, where $M^{ss}_{C,(r,d)}$ is a projective scheme \cite[Example 8.7]{Alper}.  Recall that two closed points of $\MMss_{C,(r,d)}$ are mapped to the same point of $M^{ss}_{C,(r,d)}$ if and only if they represent $S$-equivalent semistable sheaves on $C$.

Also, for any smooth curve $C$ on $\hat{X}$, we can  define a morphism of moduli stacks
\[
\alpha : \MMml_{\hat{X},\hat{ch}} \to \MMss_{C,(r,d)}
 \]
 as follows: given any scheme $S$ over $k$ and any object $F_S \in \MMml_{\hat{X},\hat{ch}} (S) \subset D(\hat{X}_S)$, let $\alpha$ send $F_S$ to the object $F_S |_{C \times S}$; here, we need $r=\hat{ch}_0$ and $d = \hat{ch}_1 \cdot [C]$.

Since the moduli stack $\MMml_{\hat{X},\hat{ch}}$ is bounded, by \cite[Theorem 8.2.12]{HL}, if the smooth curve $C$ above further satisfies the requirement $C \in |m\hat{\omega}|$ for $m \gg 0$ depending on $\hat{ch}$, then two $\mu$-semistable locally free sheaves $F_1, F_2$ of Chern character $\hat{ch}$ on $\hat{X}$ are isomorphic if and only if their restrictions $F_1 |_C, F_2|_C$ are $S$-equivalent.  That is, the composition
\[
\gamma \alpha : \MMml_{\hat{X},\hat{ch}} \to \MMss_{C,(r,d)} \to M^{ss}_{C,(r,d)}
\]
 induces an injection on the sets of closed points.  Pre-composing $\gamma\alpha$ with the morphism $\Psi_\MM$, we obtain:

 \begin{coro}\label{coro1}
When the objects in $\MMs_{X,ch}$ satisfy the conditions of Proposition \ref{pro1},  and $C \in |m\hat{\omega}|$ is a smooth curve with $m \gg 0$ depending on $\hat{ch}$, we have a morphism of stacks
 \[
\gamma \alpha \Psi_\MM  : \MMs_{X,ch} \to M^{ss}_{C,(r,d)}
 \]
 from the Artin stack of Bridgeland semistable objects $\MMs_{X,ch}$ to the projective scheme $M^{ss}_{C,(r,d)}$ of semistable torsion-free sheaves on $C$, that induces an injection on the sets of closed points.
 \end{coro}

\section{Consequences and Applications}\label{section-app}

 In  part (iii) of \cite[Theorem 5.4]{LQ}, it was shown that there is a set-theoretic bijection between the closed points of the moduli stack $\MMs_{X,ch}$ from Section \ref{section3}, and the closed points of the Uhlenbeck compactification space of $\mu$-stable locally free sheaves on $X$; this required the Bridgeland semistable objects to have rank and degree that are coprime.   Corollary \ref{coro1} gives us a morphism from the stack $\MMs_{X,ch}$ to the scheme $M^{ss}_{C,(r,d)}$, which induces an injection of the closed points, thereby giving us an algebro-geometric alternative to the statement of part (iii) of \cite[Theorem 5.4]{LQ}, and without the coprime assumption.

Parts (i) and (ii) of \cite[Theorem 5.4]{LQ} say that, whenever the rank and degree of the objects are coprime, the moduli stack of polynomial semistable objects of Chern character $ch$ on a K3 surface is always isomorphic to a moduli scheme of  stable sheaves.  Note that, parts (i) and (ii) of \cite[Theorem 5.4]{LQ} can be read as statements on isomorphic moduli stacks, for if each fibre of a  family of complexes is isomorphic to a sheaf at degree 0, then the family itself is isomorphic to a flat family of sheaves at degree 0; also, derived dual and shift are both functors that take families of complexes to families of complexes.  Together with Theorem \ref{theorem1}, we now know that under the coprime assumption on rank and degree, the moduli stack of polynomial semistable objects must be  isomorphic to either a moduli scheme of  stable sheaves, or a moduli stack of $\mu$-semistable locally free sheaves.

Also, the proof of Theorem \ref{theorem1} in fact shows, that every $\mu_{\hat{\omega}}$-semistable torsion-free sheaf on $\hat{X}$ of Chern character $\hat{ch}$ is the image of a $\sigma$-semistable object on $X$ under $\Psi$.  In other words, there are no non-locally free $\mu_{\hat{\omega}}$-semistable sheaves  of Chern character $\hat{ch}$ on $\hat{X}$.  Since the moduli stack of $\mu_{\hat{\omega}}$-semistable  sheaves is universally closed,  we conclude that the moduli stack $\MMs_{X,ch}$ in Theorem \ref{theorem1} is also universally closed.

Lastly, we point out that Theorem \ref{theorem1} can be considered as an extension of a theorem of Bruzzo-Maciocia: in \cite[Theorem 0.1]{BM}, for a reflexive K3 surface $X$, they constructed an isomorphism from the Hilbert scheme $\text{Hilb}^n(X)$ of $n$ points on $X$ to the moduli space of Gieseker stable locally free sheaves on the Fourier-Mukai partner $\hat{X}$ when $n \geq 1$.  Choosing $\beta \omega =0$ and using ideal sheaves, we can consider $\text{Hilb}^n(X)$ as a substack of the moduli stack of Bridgeland semistable objects of Chern character $(-1,0,n)$.  This way, our Theorem \ref{theorem1} extends Bruzzo-Maciocia's isomorphism by adding $\mu$-semistable locally free sheaves on the $\hat{X}$ side (also see Remark \ref{remark1}).

\end{document}